\documentclass[12pt]{article}

\usepackage{amssymb}
\usepackage{epsfig}
\usepackage{amsmath}
\usepackage{amsthm}
\usepackage{subfigure}
\usepackage{url}
\usepackage{color}
\usepackage{authblk}
\usepackage{graphicx}
\usepackage{makeidx}
\usepackage{relsize}
\usepackage{titlesec}
\usepackage{float}
\usepackage{ragged2e}
\usepackage{datetime}
\usepackage[font=footnotesize,labelfont=bf]{caption}

\usepackage[margin=1.5in]{geometry}
\DeclareMathOperator{\transverse}{\cap\kern-7.75pt\top}

\definecolor{titlered}{rgb}{.6,0,0}
\definecolor{titleblue}{rgb}{0,0,0.6}
\definecolor{exerciseblue}{rgb}{0,0,.6}
\definecolor{proofgreen}{rgb}{0.6,0,0}
\definecolor{myblue}{rgb}{0,0,0.7}
\definecolor{mymagenta}{rgb}{0.7,0,0}

\newcommand{\blue}{\color{blue}}

\newcommand{\reach}{\operatorname{reach}}

\newcommand{\argmin}{\operatornamewithlimits{argmin}}

\newcommand{\diam}{\operatorname{diam}}

\newcommand{\cls}{\operatorname{cls}}
\newcommand{\unique}{\operatorname{unique}}
\newcommand{\dist}{\operatorname{dist}}

\newcommand{\R}{\Bbb{R}}

\newcommand{\Hd}{\mathcal{H}}

\let\oldhat\hat 
 
\renewcommand{\hat}[1]{\oldhat{\mathbf{#1}}}

\newtheorem{thm}{Theorem}[section]

\newtheorem{con}{Conjecture}[section]

\newtheorem{rem}{Remark}[section]
\newtheorem{exm}{Example}[section]
\newtheorem{defn}{Definition}[section]
\newtheorem{excse}{Exercise}[section]
\newtheorem{prob}{Problem}[section]
\newtheorem{que}{Question}[section]

\usepackage[bookmarks=true, 
            bookmarksnumbered=false, 
            bookmarksopen=false, 
            colorlinks=true, 
            linkcolor=red]{hyperref}

\titleformat{\section}
{\color{titleblue}\normalfont\Large\bfseries}
{\color{titleblue}\thesection}{1em}{}

\titleformat{\subsection}
{\color{titleblue}\normalfont\large\bfseries}
{\color{titleblue}\thesubsection}{1em}{}

\begin{document}

\title{{\bf\color{titleblue}{Cubical Covers of Sets in $\R^n$}} }
\author[1]{Laramie Paxton\thanks{realtimemath@gmail.com}} 
\author[1]{Kevin R. Vixie\thanks{vixie@speakeasy.net}}
\affil[1]{Department of Mathematics and Statistics \protect\\ Washington State University}

\date{}
\renewcommand\Authands{ and }

\maketitle

\begin{abstract}
  Wild sets in $\R^n$ can be tamed through the use of various
  representations though sometimes this taming removes features
  considered important. Finding the wildest sets for which it is still
  true that the representations faithfully inform us about the
  original set is the focus of this rather playful, expository paper
  that we hope will stimulate interest in wild sets and their
  representations, as well as the other two ideas we explore very
  briefly: Jones' $\beta$ numbers and varifolds from geometric measure
  theory.
\end{abstract}

\small{\tableofcontents}

\section{Introduction}
\label{sec:intro}

In this paper we explain and illuminate a few ideas for (1)
representing sets and (2) learning from those representations. Though
some of the ideas and results we explain are likely written down
elsewhere (though we are not aware of those references), our purpose
is not to claim priority to those pieces, but rather to stimulate
thought and exploration. Our primary intended audience is students of
mathematics even though other, more mature mathematicians may find a
few of the ideas interesting. We believe that cubical covers can be
used at an earlier point in the student career and that both the
$\beta$ numbers idea introduced by Peter Jones and the idea of
varifolds pioneered by Almgren and Allard and now actively being
developed by Menne, Buet, and collaborators are still very much
underutilized by all (young and old!). To that end, we have written
this exploration, hoping that the questions and ideas presented here,
some rather elementary, will stimulate others to explore the ideas for
themselves.

We begin by briefly introducing cubical covers, Jones' $\beta$, and
varifolds, after which we look more closely at questions involving
cubical covers. Then both of the other approaches are explained in a
little bit of detail, mostly as an invitation to more exploration,
after which we close with problems for the reader and some unexplored
questions.

{\bf Acknowledgments:} LP thanks Robert Hardt and Frank Morgan for
useful comments and KRV thanks Bill Allard for useful conversations and
Peter Jones, Gilad Lerman, and Raanan Schul for introducing him to the 
idea of the Jones' $\beta$ representations.
\newpage
\section{Representing Sets \& their Boundaries in $\R^n$} 

\subsection{Cubical Refinements: Dyadic Cubes}

In order to characterize various sets in $\R^n$, we explore
the use of cubical covers whose cubes have side
lengths which are positive integer powers of $\frac{1}{2}$,
\textbf{dyadic cubes}, or more precisely, (closed) dyadic $n$-cubes
with sides parallel to the axes. Thus the side length at the $d$th subdivision
is $l(C)=\frac{1}{2^d}$, which can be made as small as desired.

Figure~\ref{fig:dyadic} illustrates this by looking at a unit cube
in $\R^2$ lying in the first quadrant with a vertex at the origin. We
then form a sequence of refinements by dividing each side length in
half successively, and thus quadrupling the number of cubes each time. 

\begin{defn}
We shall say that the $n$-cube $C$ (with side length denoted as $l(C)$) is dyadic if
\[
C=\prod^n_{j=1}[m_j2^{-d},(m_j+1)2^{-d}], \ \ \ m_j \in \mathbb{Z}, \ d\in \mathbb{N}\cup\{0\}.
\]

\end{defn}
\begin{figure} [H]
\centering

      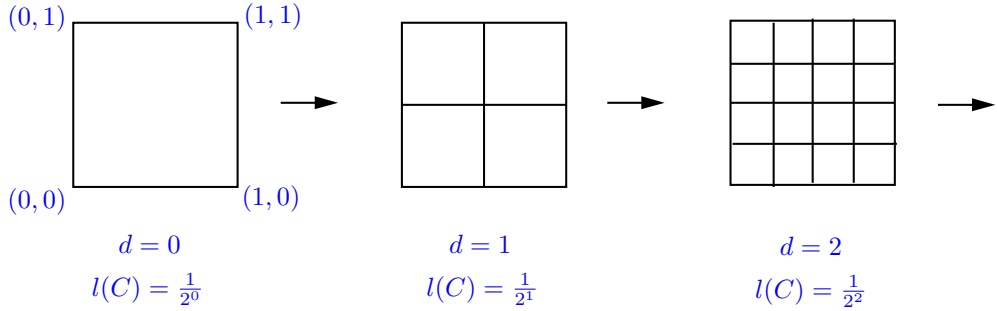
      
      \caption{Dyadic Cubes.} 
      \label{fig:dyadic}
\end{figure}

In this paper, we will assume $C$ to be a dyadic $n$-cube
throughout. We will denote the union of the dyadic $n$-cubes with edge
length $\frac{1}{2^d}$ that intersect a set $E\subset \R^n$ by
$\mathcal{C}^E_d$ and define $\partial\mathcal{C}^E_d$ to
be the boundary of this union (see Figure
\ref{fig:cubicalcover1}). Two simple questions we will explore for
their illustrative purposes are:

\begin{enumerate}
\item ``If we know $\mathcal{L}^n(\mathcal{C}^E_d)$, what can we say about $\mathcal{L}^n(E)$?'' and similarly,
\item ``If we know $\mathcal{H}^{n-1}(\partial\mathcal{C}^E_d)$, what can we say about $\mathcal{H}^{n-1}(\partial E)$?''
\end{enumerate}
\begin{figure}[H]
  \centering
  \scalebox{0.5}{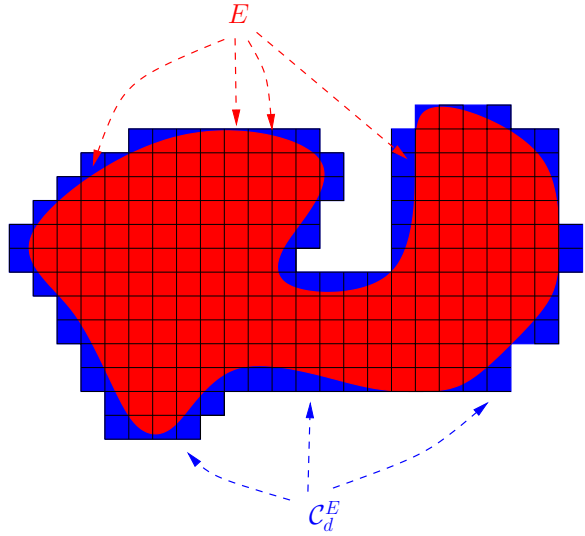}
  \caption{Cubical cover $\mathcal{C}^E_d$ of a set $E$.}
  \label{fig:cubicalcover1}
\end{figure}

\subsection{Jones' $\beta$ Numbers}

Another approach to representing sets in $\R^n$, developed by Jones
\cite{jon90}, and generalized by Okikiolu \cite{oki92}, Lerman
\cite{ler03}, and Schul \cite{sch07}, involves the question of under
what conditions a bounded set $E$ can be contained within a
rectifiable curve $\Gamma$, which Jones likened to the Traveling
Salesman Problem taken over an infinite set. (See Definition~\ref{rect}
below for the definition of rectifiable.)

Jones showed that if the aspect ratios of the optimal containing
cylinders in  each dyadic cube go to zero fast enough, the set $E$
is contained in a rectifiable curve. Jones' approach ends up providing
one useful approach of defining a representation for a set in $\R^n$
similar to those discussed in the next section. We return to this
topic in Section \ref{jonessection2}. The basic idea is illustrated in
Figure~\ref{fig:jones1}.

\begin{figure} [H]
\begin{center}
      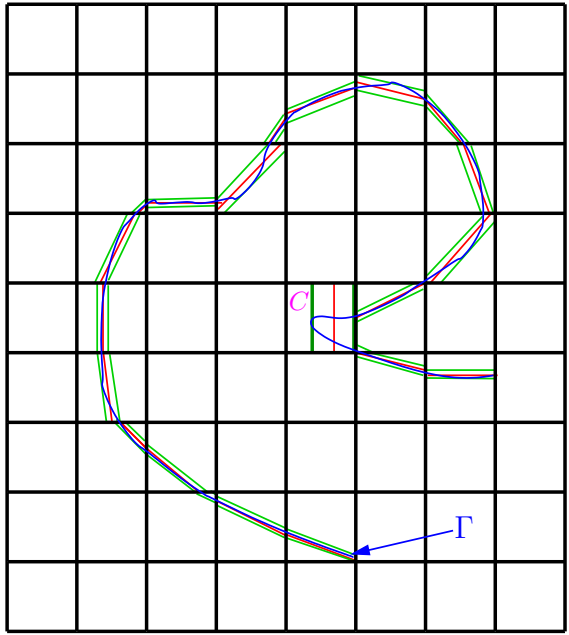
      \caption{Jones' $\beta$ Numbers. The green lines indicate the thinnest cylinder containing $\Gamma$ in the cube $C$. We see from this relatively large width that $\Gamma$ is not very ``flat" in this cube.}\label{jones1}
  \label{fig:jones1}
 \end{center}
\end{figure}

\subsection{Working Upstairs: Varifolds}\label{Var}

A third way of representing sets in $\R^n$ uses \emph{varifolds}. 
Instead of representing $E\subset \R^n$ by working in $\R^n$, we work
in the \emph{Grassmann Bundle}, $\R^n\times G(n,m)$. 

We parameterize the Grassmannian $G(2,1)$ by taking the upper unit semicircle in
$\R^2$ (including the point $(1,0)$, but not including $(1,\pi)$,
where both points are given in polar coordinates) and straightening it
out into a vertical axis (as in Figure \ref{var1a}). The bundle $\R^2
\times G(2,1)$ is then represented by $\R^2 \times [0,\pi).$
\begin{figure}[H]
\begin{center}
      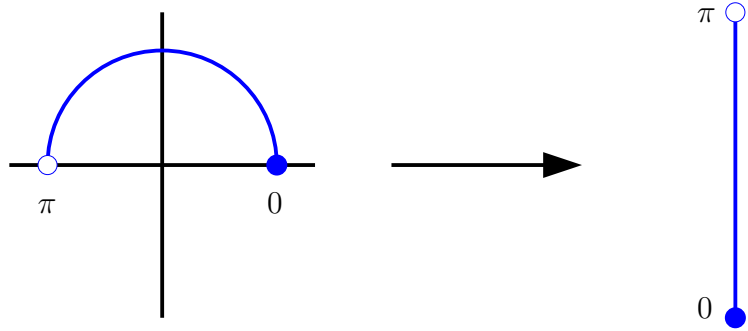
      \caption{The vertical axis for the ``upstairs."}\label{var1a}
 \end{center}
\end{figure}

\begin{figure}[H]
\begin{center}
      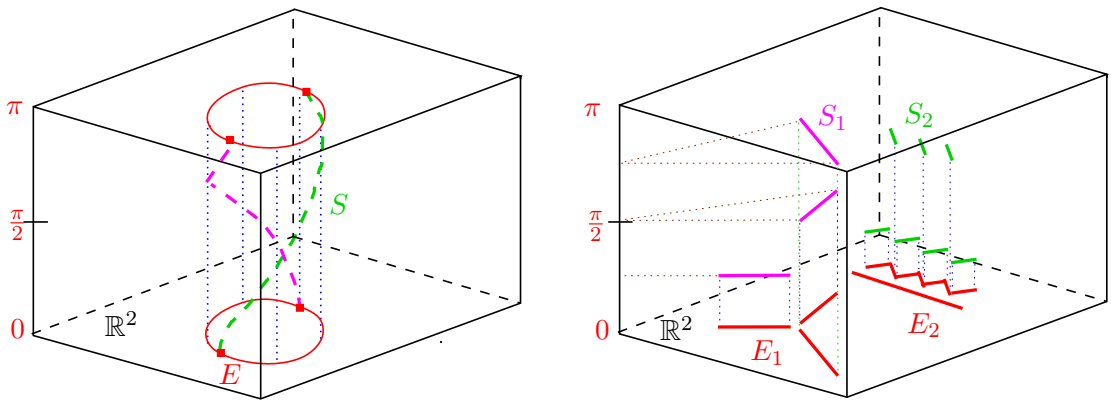
      \caption{Working Upstairs in the Grassmann bundle.}\label{var1c}
 \end{center}
\end{figure}
Figure~\ref{var1c} illustrates how the tangents are built into this
representation of subsets of $\R^n$, giving us a sense of why this
representation might be useful. A circular curve in $\R^2$ becomes two
half-spirals upstairs (in the Grassmann bundle representation, as
shown in the first image of Figure \ref{var1c}). Other curves in $\R^2$
are similarly illuminated by their  Grassmann bundle
representations. We return to this idea in Section \ref{Var2}.

\section{Simple Questions}

Let $E\subset \R^n$ and $C$ be any dyadic $n$-cube as before. Define \[\mathcal{C}(E,d) = {\{C \ | \ C\cap E \neq \emptyset, \ l(C) = {1}/{2^d}\}}\]
and, as above,
\[\mathcal{C}^E_d \equiv \bigcup_{C\in \mathcal{C}(E,d)} C. \] 
Here are two questions:
\begin{enumerate}
\item Given $E\subset\R^n$, when is there a $d_0$ such that for all $d\geq d_0$, we have
\begin{equation}\label{A}
 \mathcal{L}^n(\mathcal{C}^E_d) \leq M(n)\mathcal{L}^n(E)
\end{equation}
for some constant $M(n)$ independent of $E$?
\item Given $E\subset\R^n$, and any $\delta>0$, when does there exists
  a $d_0$ such that for all $d\geq d_0$, we have
\begin{equation}\label{B}
 \mathcal{L}^n(\mathcal{C}^E_d) \leq (1+\delta)\mathcal{L}^n(E)?
\end{equation} 
\end{enumerate}

\begin{rem} Of course using the fact that Lebesgue measure is a Radon
  measure, we can very quickly get that for $d$ large enough
  (i.e. $2^{-d}$ small enough), the measure of the cubical cover is as
  close to the measure of the set as you want, as long as the set is
  compact and has positive measure.  But the focus of this paper is on
  what we can get in a much more transparent, barehanded fashion, so we
  explore along different paths, getting answers that are, by some metrics,
  suboptimal.
\end{rem}

\begin{exm}\label{exm:rational-points}
  If $E=\mathbb{Q}^n\cap [0,1]^n$, then $\mathcal{L}^n(E)=0$, but
  $\mathcal{L}^n(\mathcal{C}^E_d)=1 \ \forall d \geq 0$.
\end{exm}
\begin{exm}
  Let $E$ be as in Example~\ref{exm:rational-points}. Enumerate $E$
  as $\hat{q_1}, \hat{q_2},\hat{q_3}, \ldots$. Now let
  $D_i=B(\hat{q_i},\frac{\epsilon}{2^i})$ and $E_\epsilon \equiv
  \{\cup D_i\} \cap [0,1]^n$ with $\epsilon$ chosen small enough so
  that $\mathcal{L}^n(E_\epsilon)\leq \frac{1}{100}$. Then
  $\mathcal{L}^n(E_\epsilon)\leq \frac{1}{100}$, but
  $\mathcal{L}^n(\mathcal{C}^{E_\epsilon}_d)=1 \ \forall d > 0$.
\end{exm}

\subsection{A Union of Balls}\label{union}

For a given set $F\subseteq \R^n,$ suppose $E= \cup_{x\in F}\bar{B}(x,r)$, a
union of closed balls of radius $r$ centered at each point $x$ in $F$.
Then we know that $E$ is \textbf{regular} (locally Ahlfors $n$-regular
or \textit{locally} $n$-regular), and thus there exist $0 < m < M <
\infty$ and an $r_0>0$ such that for all $x\in E$ and for all
$0<r<r_0$, we have
\[
m r^n \leq \mathcal{L}^n(\bar{B}(x, r)\cap E) \leq M r^n.
\]
This is all we need to establish a sufficient condition for Equation
(\ref{A}) above. 

\begin{rem}
The upper bound constant $M$ is immediate
since $E$ is a union of $n$-balls, so  $M=\alpha_n$, the
$n$-volume of the unit $n$-ball, works. However, this is not the case for
$k$-regular sets in $\R^n$, $k<n$, since we are now asking for a bound on the
$k$-dimensional measure of an $n$-dimensional set which could easily
be infinite.
\end{rem}

\begin{enumerate}
\item Suppose $E= \cup_{x\in F}\bar{B}(x,r)$, a union of closed balls of
  radius $r$ centered at each point $x$ in $F$.
\item Let $\mathcal{C}=\mathcal{C}(E,d)$ for some $d$ such that
  $\frac{1}{2^d} \ll r$, and let $\oldhat{\mathcal{C}}=\{3C \mid C\in
  \mathcal{C}\},$ where $3C$ is an $n$-cube concentric with $C$ with
  sides parallel to the axes and $l(3C)=3l(C)$, as shown in Figure
  \ref{3c}.
 \begin{figure}[H]
\begin{center}
      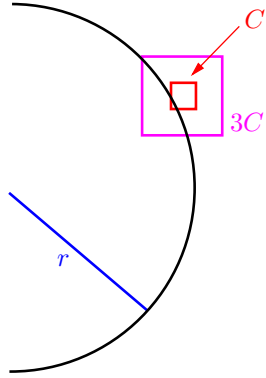
      \caption{Concentric Cubes.}\label{3C}
      \label{3c}
 \end{center}
\end{figure}
\item\label{crucial-cube-step} This implies that for $3C\in \oldhat{\mathcal{C}}$
\begin{equation}
\frac{\mathcal{L}^n(3C\cap E)}{\mathcal{L}^n(3C)}>\theta >0, \ \ \ \text{with} \  \theta \in \R.
\end{equation}
\item We then make the following observations:
\begin{enumerate}
\item Note that there are $3^n$ different tilings of the plane by $3C$
  cubes whose vertices live on the $\frac{1}{2^d}$ lattice. (This can
  be seen by realizing that there are $3^n$ shifts you can perform on
  a $3C$ cube and both (1) keep the originally central cube $C$ in the
  $3C$ cube and (2) keep the vertices of the $3C$ cube in the
  $\frac{1}{2^d}$ lattice.)
\item Denote the $3C$ cubes in these tilings $\mathcal{T}_i, i =
  1,...,3^n$.
\item Define $\oldhat{\mathcal{C}}_i \equiv \oldhat{\mathcal{C}}\cap\mathcal{T}_i$.
\item Note now that by Step~(\ref{crucial-cube-step}), the number of $3C$ cubes in
  $\oldhat{\mathcal{C}}_i$ cannot exceed
\[N_i \equiv \frac{ \mathcal{L}^n(E)}{\theta \mathcal{L}^n(3C)}.\]
\item Denote the total number of cubes in $\mathcal{C}$ by $N_{\mathcal{C}^E_d}$.
\item The number of cubes in $\mathcal{C}$,  $N_{\mathcal{C}^E_d}$, cannot exceed 
\[\sum_{i=1}^{3^n} N_i = 3^n\frac{ \mathcal{L}^n(E)}{\theta \mathcal{L}^n(3C)}.\]
\item Putting it all together, we get
\begin{eqnarray}
  \mathcal{L}^n(\mathcal{C}_d^E)  &=& \mathcal{L}^n(\cup_{C\in\mathcal{C}} C)\nonumber\\  
& = & N_{\mathcal{C}^E_d} \mathcal{L}^n(C) \nonumber\\
&\leq & 3^n \frac{\mathcal{L}^n(E) }{\theta \mathcal{L}^n(3C)} \mathcal{L}^n(C) \nonumber\\
& = &  \frac{\mathcal{L}^n(E)}{\theta}. 
\end{eqnarray}
\end{enumerate}
\item This shows that if $E= \cup_{x\in F}\bar{B}(x,r)$, then 
\[  \mathcal{L}^n(\mathcal{C}^E_d) \leq\frac{1}{\theta}\mathcal{L}^n(E).\]
\end{enumerate}
We now have two conclusions:
\begin{description}
\item[Regularized sets]We notice that for any fixed $r_0 > 0$, as long as we pick $d_0$ big enough, then $r < r_0$ and $d > d_0$ imply that $E= \cup_{x\in F}\bar{B}(x,r)$ satisfies
  \[ \mathcal{L}^n(\mathcal{C}^E_d)
  \leq\frac{1}{\theta(n)}\mathcal{L}^n(E),\] for a $\theta(n) > 0$ that
  depends on n, but not on $F$.
\item[Regular sets] Now suppose that \[F \in \mathcal{R}_m \equiv \{W\subset\R^n \;| \; mr^n < \mathcal{L}^n(W\cap \bar{B}(x,r)),\; \forall \; x \in W \text{ and } r< r_0\}.\] Then we immediately get the same result:  for a big enough $d$ (depending only on $r_0$),
  \[ \mathcal{L}^n(\mathcal{C}^F_d)
  \leq\frac{1}{\theta(m)}\mathcal{L}^n(F),\] where $\theta(m) > 0$
  depends only on the regularity class that $F$ lives in and not on
  which subset in that class we cover with the cubes.
\end{description}

\subsection{Minkowski Content}

\begin{defn}
\emph{(Minkowski content).} Let $W\subset \R^n$, and let $W_r \equiv \{x \mid d(x,W)<r\}$. The $(n-1)$-dimensional Minkowski Content is defined as $\mathcal{M}^{n-1}(W)\equiv \lim_{r \rightarrow 0} \frac{\mathcal{L}^n(W_r)}{2r}$, when the limit exists (see Figure \ref{Min}).
\end{defn}

\begin{defn}\label{rect}
\emph{(($\Hd^{m},m$)-rectifiable set)}. A set $W\subset \R^n$ is called {\bf ($\Hd^{m},m$)-rectifiable} if $\mathcal{H}^m(W)<\infty$ and $\mathcal{H}^m$-almost all of $W$ is contained in the union of the images of countably many Lipschitz functions from $\R^m$ to $\R^n$. We will use {\bf rectifiable} and {\bf ($\Hd^{m},m$)-rectifiable} interchangeably when the dimension of the sets are clear from the context.
\end{defn}

 \begin{figure}[H]
\begin{center}
\scalebox{0.5}{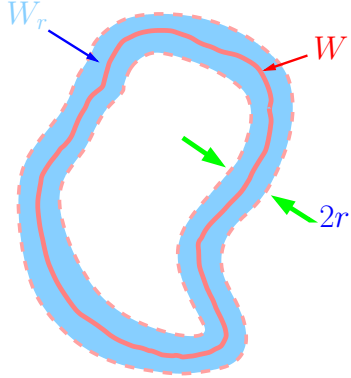}
      \caption{Minkowski Content.}\label{Min}
 \end{center}
\end{figure}

\begin{defn}[m-rectifiable]
  We will say that $E\subset\R^n$ is {\bf $m$-rectifiable} if there is a Lipschitz function mapping a bounded subset of $\R^m$ onto $E$.
\end{defn}

\begin{thm}\label{mink-equals-hausdorff}
$\mathcal{M}^{n-1}(W)=\mathcal{H}^{n-1}(W)$ when $W$ is a closed, $(n$-$1)$-rectifiable set.
\end{thm}
See Theorem 3.2.39 in \cite{fed69} for a proof.

\begin{rem}
  Notice that $m$-rectifiable is more restrictive that $(\Hd^m,m)$-rectifiable. In fact, Theorem~\ref{mink-equals-hausdorff} is false for $(\Hd^m,m)$-rectifiable sets. See the notes at the end of section  3.2.39 in \cite{fed69} for details.
\end{rem}
 
 Now, let $W$ be ($n$-$1$)-rectifiable, set $r_d \equiv
 \sqrt{n}\left(\frac{1}{2^d}\right)$, and choose $r_\delta$ small
 enough so that
\[
\mathcal{L}^n(W_{r_d}) \leq \mathcal{M}^{n-1}(W)2r_d + \delta,
\]
for all $d\in \mathbb{N}\cup\{0\}$ such that $r_d\leq r_\delta.$
(Note: Because the diameter of an $n$-cube with edge length
$\frac{1}{2^d}$ is $r_d = \sqrt{n}\left(\frac{1}{2^d}\right)$ , no point of
  $\mathcal{C}_d^W$ can be farther than $r_d$ away from $W$. Thus
  $\mathcal{C}_d^W\in W_{r_d}$.)

Assume that $\mathcal{L}^n(E) \neq 0$ and $\partial E$ is ($n$-$1$)-rectifiable. Letting $W\equiv \partial E$, we  have
\begin{eqnarray*}
\mathcal{L}^n(\mathcal{C}^E_d)-\mathcal{L}^n(E) &\leq& \mathcal{L}^n(W_{r_d})\\
& \leq& \mathcal{M}^{n-1}(\partial E)2r_d + \delta\\ 
&\leq& \mathcal{M}^{n-1}(\partial E)2r_\delta + \delta 
\end{eqnarray*}
so that
\begin{equation}\label{eq:mink-bound}
\mathcal{L}^n(\mathcal{C}^E_d)\leq (1+\oldhat{\delta})\mathcal{L}^n(E), \ \ \text{where} \ \oldhat{\delta}=\frac{\mathcal{M}^{n-1}(\partial E)2r_\delta + \delta}{\mathcal{L}^n(E)}.
\end{equation}
Since we control $r_\delta$ and $\delta$, we can make $\oldhat{\delta}$ as small as we like, and we have a sufficient condition to establish Equation (\ref{B}) above.\\

\noindent {\bf The result}: let $\oldhat{\delta}$ be as in
Equation~(\ref{eq:mink-bound}) and $E\subset \R^n$ such that
$\mathcal{L}^n(E) \neq 0$. Suppose that $\partial E$ (which is
automatically closed) is ($n$-$1$)-rectifiable and $\Hd^{n-1}(\partial
E)<\infty$, then, for every $\delta>0$ there exists a $d_0$ such that
for all $d\geq d_0$,
\[\mathcal{L}^n(\mathcal{C}^E_d)\leq (1+\oldhat{\delta})\mathcal{L}^n(E).\]

\begin{prob}
  Suppose that $E\subset\R^n$ is bounded. Show that for any $r > 0$,
  $E_r$, the set of points that are at most a distance $r$ from $E$,
  has a $(\Hd^{n-1},n-1)$-rectifiable boundary. Show this by showing
  that $\partial E_r$ is contained in a finite number of graphs of
  Lipschitz functions from $\R^{n-1}$ to $\R$. Hint: cut $E_r$ into
  small chunks $F_i$ with common diameter $D \ll r$ and prove that
  $(F_i)_r$ is the union of a finite number of Lipschitz graphs.
\end{prob}

\begin{prob}
  Can you show that in fact the boundary of $E_r$, $\partial E_r$, is
  actually ($n$-$1$)-rectifiable? See if you can use the results of the
  previous problem to help you.
\end{prob}

\begin{rem}
  We can cover a union $E$ of open balls of radius $r$, whose centers
  are bounded, with a cover $\mathcal{C}^E_d$ satisfying
  Equation~(\ref{B}). In this case, $\partial \mathcal{C}^E_d$
  certainly meets the requirements for the result just shown.
\end{rem}

\subsection{Smooth Boundary, Positive Reach}
\label{sec:reach}

In this section, we show that if $\partial E $ is \textit{smooth} (at
least $C^{1,1}$), then $E$ has positive \textit{reach} allowing us to
get an even cleaner bound, depending in a precise way on the curvature
of $\partial E$. 

We will assume that $E$ is closed. Define $E_r = \{x\in
R^n \, | \, \dist(x,E) \leq r\}$, $\cls(x) \equiv \{y\in E \; | \; d(x,E) = |x-y|\}$
and $\unique(E) = \{x \; | \, \cls(x)\text{ is a single point}\}$.

\begin{defn}[Reach]
  The  \textbf{reach} of $E$, $\reach(E)$, is defined
\[\reach(E)\equiv \sup \{r \; | \;  E_r \subset \unique(E)\}\]
 \end{defn}

 \begin{rem}
Sets of positive reach were introduced by Federer in 1959~\cite{federer1959curvature} in a paper that also introduced the famous coarea formula.    
 \end{rem}

 \begin{rem}
If $E\subset\R^n$ is ($n$-$1$)-dimensional and $E$
 is closed, then $E = \partial E$.   
 \end{rem}

Another equivalent definition involves rolling balls around the boundary of $E$.
The closed ball $\bar{B}(x,r)$ {\bf touches} $E$ if \[\bar{B}(x,r)\cap E \subset \partial\bar{B}(x,r)\cap\partial E\]
 \begin{defn}
 The  \textbf{reach} of $E$, $\reach(E)$, is defined
\[\reach(E)\equiv \sup \{r \; |   \text{ every ball of radius $r$ touching $E$ touches at a single point}\}\]
 \end{defn}
 Put a little more informally, $\reach(E)$ is the supremum of
 radii $r$ of the balls such that each ball of that radius rolling around E touches E at only one point (see Figure \ref{reach}). 

\begin{figure}[H]
\begin{center}
      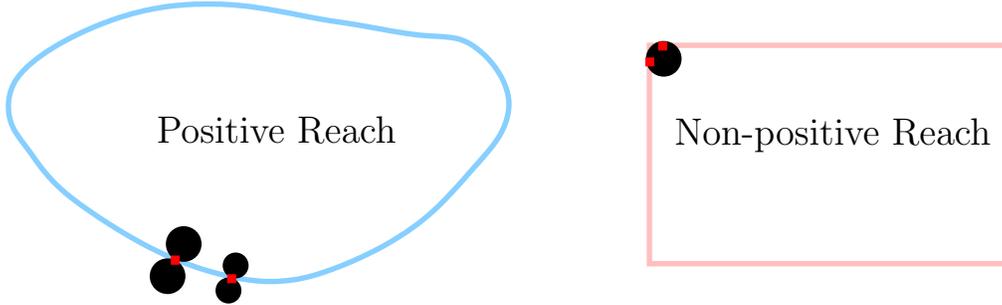
      \caption{Positive and Non-positive Reach.}\label{reach}
 \end{center}
\end{figure}

As mentioned above, if $\partial E$ is $C^{1,1}$, then it has positive
reach (see Remark 4.20 in \cite{federer1959curvature}). That is, if
for all $x\in \partial E$, there is a neighborhood of $x$, $U_x\subset
\R^n$, such that after a suitable change of coordinates, there is a
$C^{1,1}$ function $f:\R^{n-1}\rightarrow\R$ such that $\partial E\cap
U_x$ is the graph of $f$. (Recall that a function is $C^{1,1}$ if its
derivative is Lipschitz continuous.)  This implies, among other
things, that the (symmetric) second fundamental form of $\partial E$
exists $\Hd^{n-1}$-almost everywhere on $\partial E$. The fact that
$\partial E$ is $C^{1,1}$ implies that at $\Hd^{n-1}$-almost every
point of $\partial E$, the $n-1$ principal curvatures $\kappa_i$ of
our set exist and $|\kappa_i| \leq \frac{1}{\reach(\partial E)}$ for
$1\leq i \leq n-1$.

We will use this fact to determine a bound for the $(n-1)$-dimensional
change in area as the boundary of our set is expanded outwards or
contracted inwards by $\epsilon$ (see Figure \ref{lep}, Diagram
1). Let us first look at this in $\R^2$ by examining the following
ratios of lengths of expanded or contracted arcs for sectors of a ball
in $\R^2$ as shown in Diagram 2 in Figure \ref{lep} below.

 \begin{figure}[H]
\begin{center}
      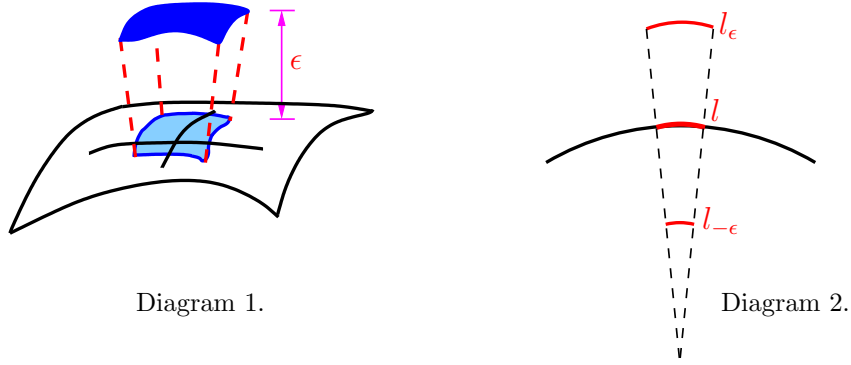
      \caption{Moving Out and Sweeping In.}\label{lep}
 \end{center}
\end{figure}
 
\[
\frac{\mathcal{H}^1(l_\epsilon)}{\mathcal{H}^1(l)}=\frac{(r+\epsilon)\theta}{r\theta}=1+\frac{\epsilon}{r}=1+\epsilon \kappa
\]
\[
\frac{\mathcal{H}^1(l_{-\epsilon})}{\mathcal{H}^1(l)}=\frac{(r-\epsilon)\theta}{r\theta}=1-\frac{\epsilon}{r}=1-\epsilon \kappa,
\]
where $\kappa$ is the principal curvature of the circle (the boundary of the 2-ball), which we can think of as  defining the reach of a set $E\subset \R^2$ with $C^{1,1}$-smooth boundary. \\

The Jacobian for the normal map pushing in or out by $\epsilon$,
which by the area formula is the factor by which the area
changes, is given by $\prod_{i=1}^{n-1}(1\pm \epsilon \kappa_i)$ (see
Figure~\ref{lep}, Diagram 1). If we define $\oldhat{\kappa}\equiv
\max\{|\kappa_1|,|\kappa_2|,\ldots, |\kappa_{n-1}|\}$, then 
we have the following ratios: 
\[
\text{Max Fractional Increase of $\mathcal{H}^{n-1}$ boundary ``area" Moving Out:}  
\]
\[
\prod_{i=1}^{n-1}(1+\epsilon \kappa_i) \leq (1+\epsilon \oldhat{\kappa})^{n-1}.
\]

\[
\text{Max Fractional Decrease of $\mathcal{H}^{n-1}$ boundary ``area" Sweeping In:}  
\]
\[
\prod_{i=1}^{n-1}(1-\epsilon \kappa_i) \geq (1-\epsilon \oldhat{\kappa})^{n-1}.
\]
\begin{rem}
  Notice that $\oldhat{\kappa} = \frac{1}{\reach(\partial E)}$.
\end{rem}

For a ball, we readily find the value of the ratio

\begin{eqnarray}
  \label{eq:ball-ratio}
\frac{\mathcal{L}^n(B(0,r+\epsilon))}{\mathcal{L}^n(B(0,r))} &=& \left( \frac{r+\epsilon}{r} \right)^n\nonumber \\
& = &   (1+\epsilon\kappa)^n \ \ \text{ (setting $\delta = \epsilon\kappa$)}\nonumber \\
& = &   (1+\delta)^n,
\end{eqnarray}
where $\kappa = \frac{1}{r}$ is the curvature of the ball along any geodesic.

Now we calculate the bound we are interested in for $E$, assuming
$\partial E$ is $C^{1,1}$. Define $E_\epsilon \subset \R^n \equiv \{x
\mid d(x,E)<\epsilon\}$.  We first compute a bound for

\begin{align}
\frac{\mathcal{L}^n(E_\epsilon)}{\mathcal{L}^n(E)} &= \frac{\mathcal{L}^n(E)+\mathcal{L}^n(E_\epsilon \setminus E)}{\mathcal{L}^n(E)} \nonumber \\
&= 1 + \frac{\mathcal{L}^n(E_\epsilon \setminus E)}{\mathcal{L}^n(E)}. \label{ratio}
\end{align}

Since $\kappa_i$ is a function of $x\in \partial E$ defined $\Hd^{n-1}$-almost everywhere, we may set up the integral below over $\partial E$ and do the actual computation over $\partial E \setminus  K$, where $K\equiv\{$the set of measure 0 where $\kappa_i $ is not defined$\}$. Computing bounds for the numerator and denominator separately in the second term in (\ref{ratio}), we find, by way of the Area Formula \cite{morgan-2008-geometric},

\begin{align}
\mathcal{L}^n(E_\epsilon \setminus E) &= \int^\epsilon_0 \int_{\partial E} \prod_{i=1}^{n-1}(1+r \kappa_i)d\mathcal{H}^{n-1}dr \nonumber \\
&\leq \int^\epsilon_0 \int_{\partial E} (1+r \oldhat{\kappa})^{n-1}d\mathcal{H}^{n-1}dr  \nonumber \\
&= \mathcal{H}^{n-1}(\partial E) \left. \frac{(1+r \oldhat{\kappa})^n}{n\oldhat{\kappa}}  \right\vert^\epsilon_0 \nonumber \\
&= \mathcal{H}^{n-1}(\partial E) \left( \frac{(1+\epsilon \oldhat{\kappa})^n}{n\oldhat{\kappa}}-\frac{1}{n\oldhat{\kappa}}\right) \label{numbound}
\end{align}
and
\begin{align}
\mathcal{L}^n(E) &\geq \int^{r_0}_0 \int_{\partial E} \prod_{i=1}^{n-1}(1-r \kappa_i)d\mathcal{H}^{n-1}dr \nonumber \\
&\geq \int^{r_0}_0 \int_{\partial E} (1-r \oldhat{\kappa})^{n-1}d\mathcal{H}^{n-1}dr  \nonumber \\
&= \mathcal{H}^{n-1}(\partial E) \left. \frac{-(1-r \oldhat{\kappa})^n}{n\oldhat{\kappa}}  \right\vert^{r_0}_0 \nonumber \\
&= \frac{\mathcal{H}^{n-1}(\partial E)}{n\oldhat{\kappa}}, \ \ \ \text{when} \ r_0=\frac{1}{\oldhat{\kappa}}. \label{denombound}
\end{align}

From \ref{ratio}, \ref{numbound}, and \ref{denombound}, we have
\begin{align}
\frac{\mathcal{L}^n(E_\epsilon)}{\mathcal{L}^n(E)} &\leq 1+ \frac{\mathcal{H}^{n-1}(\partial E) \left( \frac{(1+\epsilon \oldhat{\kappa})^n}{n\oldhat{\kappa}}-\frac{1}{n\oldhat{\kappa}}\right)}{\frac{\mathcal{H}^{n-1}(\partial E)}{n\oldhat{\kappa}}} \nonumber \\
&= (1+\epsilon \oldhat{\kappa})^n \ \  \ (\text{setting} \ \delta = \epsilon \oldhat{\kappa} \nonumber) \\
&= (1+\delta)^n.
\end{align}
From this we get that
\[ \mathcal{L}^n(E_\epsilon) \leq  (1+\epsilon\oldhat{\kappa})^n  \mathcal{L}^n(E)\]
so that 
\[\mathcal{L}^n(\mathcal{C}_{d(\epsilon)}^E) \leq (1+\epsilon\oldhat{\kappa})^n  \mathcal{L}^n(E) \]
where $d(\epsilon) = \log_2(\frac{\sqrt{n}}{\epsilon})$ is found by
solving $\sqrt{n}\frac{1}{2^d} = \epsilon$.

Thus, when $\partial E$ is smooth enough to have positive reach, we
find a nice bound of the type in Equation~(\ref{B}), with a precisely
known dependence on curvature.

\section{A Boundary Conjecture}

What can we say about boundaries? Can we bound
\[
\frac{\mathcal{H}^{n-1}(\partial \mathcal{C}^E_d)}{\mathcal{H}^{n-1}(\partial E)}?
\]

 \begin{figure}[H]
\begin{center}
      \scalebox{1.0}{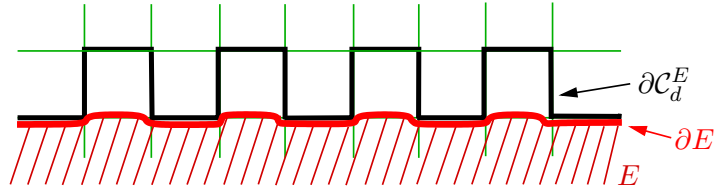}
      \caption{Cubes on the Boundary.}\label{cube}
 \end{center}
\end{figure}
 
\begin{con}\label{con}
If $E\subset \R^n$ is compact and $\partial E$ is $C^{1,1}$, 
\[
\limsup_{d \rightarrow \infty} \frac{\mathcal{H}^{n-1}(\partial \mathcal{C}^E_d)}{\mathcal{H}^{n-1}(\partial E)}\leq n.
\]
\end{con}

\begin{proof}[Brief Sketch of Proof for $n=2$] $ $
  \begin{enumerate}
 \item Since $\partial E$ is $C^{1,1}$, we can zoom in far enough
    at any point $x\in\partial E$ so that it looks flat.
  \item Let $C$ be a cube in the cover $\mathcal{C}(E,d)$ that intersects
    the boundary near $x$ and has faces in the boundary
    $\partial\mathcal{C}^E_d$. Define $F = \partial C \cap \partial
    \mathcal{C}^E_d$.
  \item (Case 1) Assume that the tangent at $x$, $T_x\partial E$, is
    not parallel to either edge direction of the cubical cover (see Figure~\ref{cube2}).
    \begin{enumerate}
    \item Let $\Pi$ be the projection onto the horizontal axis and notice
      that $\frac{\Hd^1(F)}{\Pi(F)} \leq 2+\epsilon$ for any epsilon.
    \item This is stable to perturbations which is important since the actual piece
  of the boundary $\partial E$ we are dealing with is not a straight line.
    \end{enumerate}
    \item (Case 2) Suppose that the tangent at $x$, $T_x\partial E$,
      is parallel to one of the two faces of
      the cubical cover, and let $U_x$ be a neighborhood of
      $x\in \partial E$.
    \begin{enumerate}
    \item Zooming in far enough, we see that the cubical boundary can
      only oscillate up and down so that the maximum ratio for any
      horizontal tangent is (locally) $2$.
    \item But we can create a sequence of examples that attain ratios
      as close to 2 as we like by finding a
      careful sequence of perturbations that attains a ratio locally
      of $2-\epsilon$ for any $\epsilon$ (see Figure~\ref{cube}).
    \item That is, we can create perturbations that, on an unbounded
      set of $d$'s, $\{d_i\}_{i=1}^\infty,$ yield a ratio
      $\frac{\Hd^1(\mathcal{C}^E_{d_i}\cap \, U_x)}{\partial E} >
      2-\epsilon$, and we can send $\epsilon\rightarrow 0$.
    \end{enumerate}
   \item Use the compactness of $\partial E$ to put this all together
    into a complete proof.
  \end{enumerate}
\end{proof}

 \begin{figure}[H]
\begin{center}
      \scalebox{0.75}{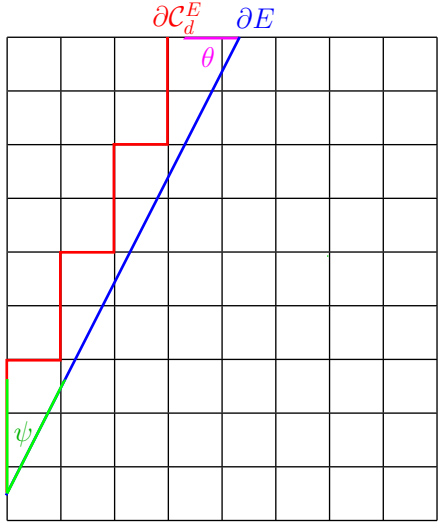}
      \caption{The case in which $\theta$, the angle between $T_x \partial E$  and the $x$-axis, is neither $0$ nor $\pi/2$.}\label{cube2}
 \end{center}
\end{figure}

\begin{prob} Suppose we exclude $C$'s that contain less than some
  fraction $\theta$ of $E$ (as defined in Conjecture \ref{con}) from the cover to get the reduced cover
  $\hat{\mathcal{C}}_d^E$.  In this case, what is the optimal bound
  $B(\theta)$ for the ratio of boundary measures
    \[\limsup_{d\rightarrow\infty}\frac{\Hd^{n-1}(\partial
      \hat{\mathcal{C}}_d^E)}{\Hd^{n-1}(\partial E)} \leq B(\theta)?\] 
    \end{prob}

\section{Other Representations}

\subsection{The Jones' $\beta$ Approach}\label{jonessection2}

As mentioned above, another approach to representing sets in $\R^n$,
developed by Jones \cite{jon90}, and generalized by Okikiolu
\cite{oki92}, Lerman \cite{ler03}, and Schul \cite{sch07}, involves
the question of under what conditions a bounded set $E$ can be
contained within a rectifiable curve $\Gamma$, which Jones likened to
the Traveling Salesman Problem taken over an infinite set. While Jones
worked in $\mathbb{C}$ in his original paper, the work of Okikiolu,
Lerman, and Schul extended the results to $\R^n \ \forall n\in
\mathbb{N}$ as well as infinite dimensional space.

Recall that a compact, connected set $\Gamma \subset \R^2$ is
rectifiable if it is contained in the image of a countable set of
Lipschitz maps from $\R$ into $\R^2$, except perhaps for a set of
$\Hd^1$ measure zero. We have the result that if $\Gamma$ is compact
and connected, then $l(\Gamma)=\mathcal{H}^1(\Gamma)<\infty$ implies it
is rectifiable (see pages 34 and 35 of
\cite{falconer-1986-geometry}). 

Let $W_C$ denote the width of the thinnest cylinder containing the set $E$ in the dyadic $n$-cube $C$ (see Figure \ref{jones2}), and define
the $\beta$ number of $E$ in $C$ to be
\[
\beta_E(C)\equiv \frac{W_C}{l(C)}.
\]
\begin{figure}[H]
\begin{center}
\scalebox{.6}{
      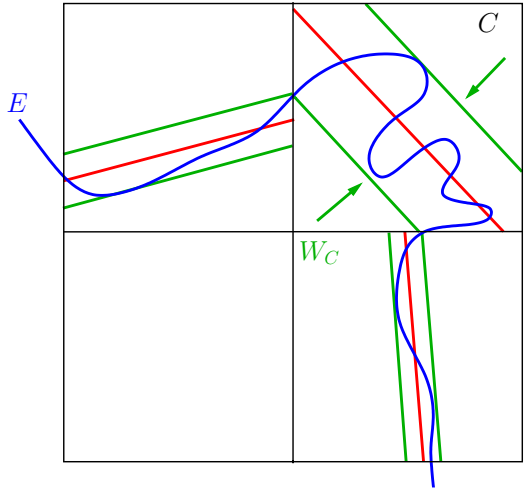}
      \caption{Jones' $\beta$ Numbers and $W_C$. Each of the two green lines in a cube $C$ is an equal distance away from the red line and is chosen so that the green lines define the thinnest cylinder containing $E\cap C$. Then the red lines are varied over all possible lines in  $C$ to find that red line whose corresponding cylinder is the thinnest of all containing cylinders. In this sense, the minimizing red lines are the best fit to $E$ in each $C$.}\label{jones2}
 \end{center}
\end{figure}
Jones' main result is this theorem:
\begin{thm}\cite{jon90}\label{jones}
  Let $E$ be a bounded set and $\Gamma$ be a connected set both in
  $\R^2$.  Define $\beta_\Gamma(C)\equiv \frac{W_C}{l(C)},$ where
  $W_C$ is the width of the thinnest cylinder in the $2$-cube $C$
  containing $\Gamma.$ Then, summing over all possible
  $C$, $$\beta^2(\Gamma)\equiv \sum_C (\beta_\Gamma(3C))^2l(C)<\eta
  \,l(\Gamma)<\infty, \text{ where } \eta \in \R.$$ \\ Conversely, if
  $\beta^2(E)<\infty$ there is a connected set $\Gamma,$ with $E
  \subset \Gamma,$ such that \[l(\Gamma) \leq (1+ \delta) \diam(E) +
  \alpha_\delta\beta^2(E),\] \text{ where } $\delta>0$ \text{ and }
  $\alpha_\delta = \alpha(\delta)\in \R.$
\end{thm}

Jones' main result, generalized to $\R^n$, is that a bounded set $E \subset \R^n$ is contained in a rectifiable curve $\Gamma$ if and only if 
\[
\beta^2(E)\equiv \sum_C (\beta_E(3C))^2l(C)<\infty,
\]
where the sum is taken over all dyadic cubes. 

Note that each $\beta$ number of $E$ is calculated over the dyadic
cube $3C$, as defined in Section \ref{union}.  Intuitively, we see
that in order for $E$ to lie within a rectifiable curve $\Gamma$, $E$
must look flat as we zoom in on points of $E$ since $\Gamma$ has
tangents at $\mathcal{H}^1$-almost every point $x\in\Gamma$. Since
both $W_C$ and $l(C)$ are in units of length, $\beta_E(C)$ is a scale-invariant measure of 
the flatness of $E$ in $C$. In higher dimensions, the
analogous cylinders' widths and cube edge lengths are also divided to
get a scale-invariant $\beta_E(C)$.

The notion of local linear approximation has been explored by many
researchers. See for example the work of  Lerman and
collaborators~\cite{chen-2009-spectral,arias-2011-spectral,zhang-2012-hybrid,arias-2017-spectral}. While
distances other than the sup norm have been considered when
determining closeness to the approximating line, see~\cite{ler03},
there is room for more exploration there. In the section below, 
\emph{Problems and Questions}, we suggest an idea involving the
multiscale flat norm from geometric measure theory.

\subsection{A Varifold Approach}\label{Var2}

As mentioned above, a third way of representing sets in $\R^n$ uses
\emph{varifolds}.  Instead of representing $E\subset \R^n$ by working
in $\R^n$, we work in the \emph{Grassmann Bundle}, $\R^n\times
G(n,m)$. Advantages include, for
example, the automatic encoding of tangent information directly into
the representation. By building into the representation this tangent
information, we make set comparisons where we care about tangent
structure easy and natural.
\begin{defn}[Grassmannian]
The $m$-dimensional Grassmannian in $\R^n$, \[G(n,m) =G(\R^n,m),\]
is the set of all $m$-dimensional planes through the origin.
\end{defn}
For example, $G(2,1)$  is the space of all lines through the origin
in $\R^2$, and $G(3,2)$ is the space of all planes through the origin
in $\R^3$. The Grassmann bundle $\R^n \times G(n,m)$ can be thought
of as a space where $G(n,m)$ is attached to each point in $\R^n$.

\begin{defn}[Varifold]
A varifold is a \emph{Radon measure} $\mu$ on the
Grassmann bundle $\R^n \times G(n,m)$.
\end{defn}

Suppose $\pi:(x,g)\in \R^n \times G(n,m) \rightarrow x$.  One of the
most common appearances of varifolds are those that arise from
rectifiable sets $E$. In this case the measure $\mu_E$ on $\R^n
\times G(n,m)$ is the pushforward of $m$-Hausdorff measure on $E$ by
the tangent map $T:x\rightarrow (x,T_xE)$.

Let $E\subset \R^n$ be an ($\Hd^{m},m$)-rectifiable set (see
Definition \ref{rect}). We know the approximate $m$-dimensional
tangent space $T_xE$ exists $\mathcal{H}^m$-almost everywhere since
$E$ is ($\Hd^{m},m$)-rectifiable, which in turn implies that, except
for an $\Hd^m$-measure 0 set, $E$ is contained in the union of the images of
countably many Lipschitz functions from $\R^m$ to $\R^n$.

The measure of $A\subset \R^n \times G(n,m)$ is given by $\mu(A)
  = \Hd^m(T^{-1}\{A\})$.  Let $S\equiv \{(x, T_xE) \, \vert \, x\in
  E\}$, the section of the Grassmann bundle defining the
  varifold. $S$, intersected with each fiber $\{x\}\times
  G(n,m)$, is the single point $(x, T_xE)$, and so we could just as
  well use the projection $\pi$ in which case we would have $\mu_E(A)
  = \Hd^m(\pi(A\cap S))$.
\begin{defn}
A \emph{rectifiable varifold} is a radon measure $\mu_E$ defined on an ($\Hd^{m},m$)-rectifiable set $E\subset \R^n$. Recalling $S\equiv \{(x, T_xE) \, \vert \, x\in E\}$, let $A \subset \R^n \times G(n,m)$ and define
\[
\mu_E({A})= \Hd^m(\pi(A \cap S)).
\]
\end{defn}

We will call $E=\pi (S)$ the ``downstairs'' representation of $S$ for
any $S\subset\R^n \times G(n,m)$, and we will call $S =
T(E)\subset\R^n\times G(n,m)$ the ``upstairs'' representation of any
rectifiable set $E$, where $T$ is the tangent map over the rectifiable
set $E$.

\bigskip

\begin{figure}[H]
\begin{center}
      \input{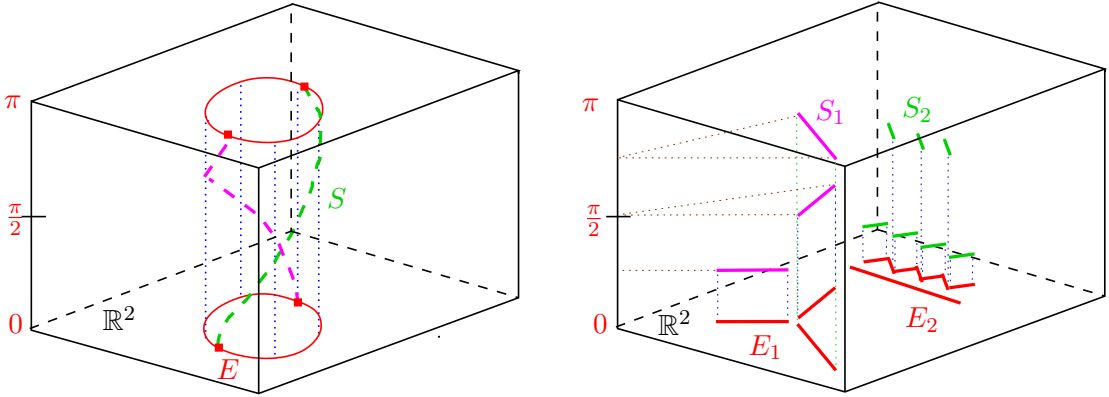}
      \caption{Working Upstairs.}\label{var1ctoo}
 \end{center}
\end{figure}

Figure~\ref{var1ctoo}, repeated from above, illustrates how the
tangents are built into this representation of subsets of $\R^n$,
giving us a sense of why this representation might be useful.  Suppose
we have three line segments almost touching each other, i.e. appearing
to touch as subsets of $\R^2$. The upstairs view puts each segment at
a different height corresponding to the angle of the segment. So,
these segments are not close in any sense in $\R^2\times G(n,m)$. Or
consider a straight line segment and a very fine sawtooth curve that
may look practically indistinguishable, but will appear drastically
different upstairs.

We can use varifold representations in combination with a cubical
cover to get a quantized version of a curve that has tangent
information as well as position information. If, for example, we cover
a set $S \subset \R^2 \times G(2,1)$ with cubes of edge length
$\frac{1}{2^d}$ and use this cover as a representation for $S$, we
know the position and angle to within $\frac{\sqrt{3}}{2^{d+1}}$. In other
words, we can approximate our curve $S \subset \R^2 \times G(2,1)$ by the
union of the centers of the cubes (with edge length $\frac{1}{2^d}$)
intersecting $S$. This simple idea seems to merit further exploration.

\section{Problems and Questions}

\begin{prob}
Find a smooth $\partial E$, with $E \subset \R^n$, such that
\[
\mathcal{H}^{n-1}(\partial \mathcal{C}^E_d) / \mathcal{H}^{n-1}(\partial E)=0 \ \forall d.
\]
\end{prob}
\textbf{Hint:} Look at unbounded $E \subset\R^2$ such that $\mathcal{L}^2(E^c) < \infty$.

\begin{prob}
Suppose that $E$ is open and $\Hd^{n-1}(\partial E) < \infty$. Show that if the \textbf{reach} of $\partial E $ is positive, then
\[
\liminf_{d \rightarrow \infty} \frac{\mathcal{H}^{n-1}(\partial \mathcal{C}^E_d)}{\mathcal{H}^{n-1}(\partial E)}\geq 1.
\]
\end{prob} \textbf{Hint:} First show that $\partial E$ has
  unique inward and outward pointing normals. (Takes a bit of work!)
  Next, examine the map $F: \partial E \rightarrow \R^n$, where
  $F(x)=x+\eta(x) N(x)$, $N(x)$ is the normal to $\partial E$ at $x$,
  and $\eta(x)$ is a positive real-valued function chosen so that
  \emph{locally} $F(\partial E) = \partial \mathcal{C}^E_d$. Use the
  Binet-Cauchy Formula to find the Jacobian, and then apply the Area
  Formula. To do this calculation, notice that at any point
  $x_0\in\partial E$ we can choose coordinates so that
  $T_{x_0}\partial E$ is horizontal (i.e. $N(x_0) = e_n$).  Calculate
  using $F: T_{x_0}\partial E = \R^{n-1} \rightarrow \R^n$ where
  $F(x)=x+\eta(x) N(x)$. (See Chapter 3 of \cite{evans-1992-1} for the
  Binet-Cauchy formula and the Area Formula.)
\begin{prob}
Suppose $E$ has dimension $n-1$, positive reach, and is locally regular (in $\R^n$). \\
a.) Find bounds for $\mathcal{H}^{n}( \mathcal{C}^E_d) / \frac{1}{2^d}.$\\
b.) How does this ratio relate to $\mathcal{H}^{n-1}(E)$?
\end{prob}
\textbf{Hint:} Use the ideas in Section~\ref{sec:reach} to calculate a
bound on the volume of the tube with thickness $2
\frac{\sqrt{n}}{2^d}$ centered on $E$.

\begin{que}
  Can we use the ``upstairs'' version of cubical covers to find better
  representations for sets and their boundaries? (Of course, ``better"
  depends on your objective!)
\end{que}

For the following question, we need the notion of the \emph{multiscale
  flat norm}~\cite{morgan-2007-1}. The basic idea of this distance,
which works in spaces of oriented curves and surfaces of any dimension
(known as currents), is that we can decompose the curve or surface $T$
into $(T - \partial S) + \partial S$, \emph{but} we measure the cost
of the decomposition by adding the volumes of $T-\partial S$ and $S$
(not $\partial S$!). By volume, we mean the $m$-dimensional volume, or
$m$-volume of an $m$-dimensional object, so if $T$ is $m$-dimensional,
we would add the $m$-volume of $T-\partial S$ and the ($m$+1)-volume
of $S$ (scaled by the parameter $\lambda$). We get that

\[\Bbb{F}_\lambda(T) = \min_S M_m(T-\partial S) + \lambda
M_{m+1}(S).\] It turns out that $T-\partial S$ is the best
approximation to $T$ that has curvature bounded by
$\lambda$~\cite{allard-2007-1}.  We exploit this in the following ideas
and questions.

\begin{rem}
  Currents can be thought of as generalized oriented curves or
  surfaces of any dimension $k$. More precisely, they are members of
  the dual space to the space of $k$-forms. For the purposes of this
  section, thinking of them as (perhaps unions of pieces of) oriented
  $k$-dimensional surfaces $W$, so that $W$ and $-W$ are simply
  oppositely oriented and cancel if we add them, will be
  enough to understand what is going on. For a nice introduction to
  the ideas, see for example the first few chapters of
  \cite{morgan-2008-geometric}.
\end{rem}

\begin{que}

Choose $k\in \{1,2,3\}$. In what follows we focus on sets $\Gamma$
which are one-dimensional, the interior of a cube $C$ will be denoted
$C^o$, and we will work at some scale $d$, i.e. the edge length of the cube 
will be $\frac{1}{2^d}$. 

Consider the piece of $\Gamma$ in $C^o$, $\Gamma \cap C^o$. Inside the
cube $C$ with edge length $\frac{1}{2^d}$, we will use the flat norm
to
\begin{enumerate}
\item find an approximation of $\Gamma \cap C^o$ with curvature 
  bounded by $\lambda = 2^{d+k}$ and
\item find the distance of that approximation from $\Gamma \cap C^o$.
\end{enumerate}
This decomposition is then  obtained by minimizing \[M_1((\Gamma\cap
C^o)-\partial S) + 2^{d+k}M_{2}(S) = \mathcal{H}^1((\Gamma\cap
C^o)-\partial S) + 2^{d+k}\mathcal{L}^2(S).\] The minimal $S$ will be
denoted $S_d$ (see Figure~\ref{fig:cube-flat-beta-1}).

\begin{figure}[H]
  \centering
  \scalebox{0.5}{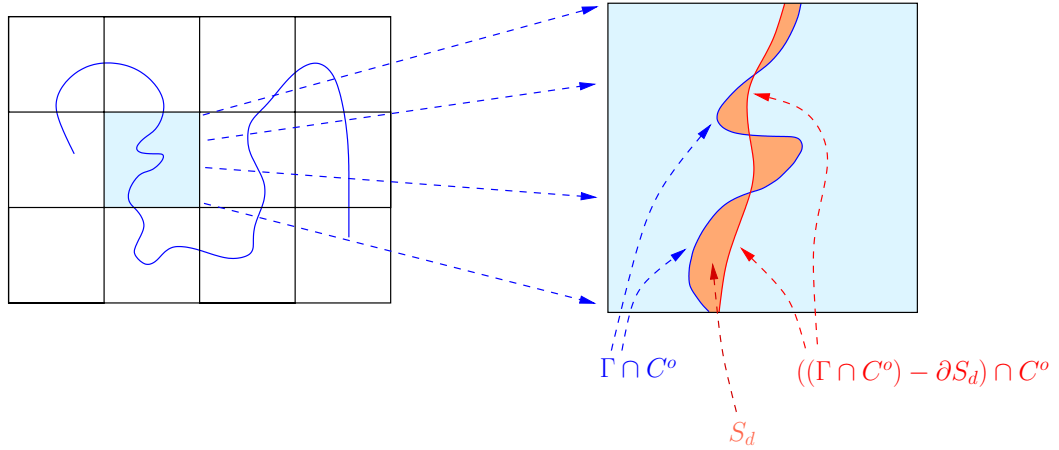}
  \caption{Multiscale flat norm decomposition inspiring the definition of $\beta_\Gamma^{\Bbb{F}}$.}
  \label{fig:cube-flat-beta-1}
\end{figure}

{\blue Suppose that we define $\beta^{\mathbb{F}}_\Gamma(C)$
  by \[\beta^{\mathbb{F}}_\Gamma(C) l(C) = 2^{d+k}\mathcal{L}^2(S_d)\]
  so that \[\beta^{\mathbb{F}}_\Gamma(C) =
  2^{2d+k}\mathcal{L}^2(S_d).\] What can we say about the properties
  (e.g. rectifiability) of $\Gamma$ given the finiteness of $\sum_C
  \left(\beta^{\mathbb{F}}_\Gamma(3C)\right)^2 l(C)$?}
\end{que}

\begin{que}
  Can we get an advantage by using the flat norm decomposition as a preconditioner before
  we find cubical cover approximations? For example, define
\[ 
\mathcal{F}_d^\Gamma \equiv \mathcal{C}^{\Gamma_d}_d \text{ and }  \Gamma_d \equiv \Gamma - \partial S_d,
\]
\[
\text{where} \ S_d = \argmin_S \left(\mathcal{H}^1(\Gamma-\partial S)+2^{d+k}\mathcal{L}^2(S)\right).
\]
Since the flat norm minimizers have bounded mean curvature, is this
enough to force the cubical covers to give us better quantitative
information on $\Gamma$? How about in the case in which $\Gamma
= \partial E$, $E\subset \R^2$?
\end{que}

\section{Further Exploration}
\label{sec:further-exploration}

There are a number of places to begin in exploring these ideas
further.  Some of these works require significant dedication to master,
and it is always recommended that you have someone who has mastered a
path into pieces of these areas that you can ask questions of when you
first wade in. Nonetheless, if you remember that the language can
always be translated into pictures, and you make the effort to do that,
headway towards mastery can always be made. Here is an annotated list
with comments:

\begin{description}
\item[Primary Varifold References] Almgren's little
  book~\cite{almgren-2001-1} and Allard's founding
  contribution~\cite{allard-1972-1} are the primary sources for
  varifolds. Leon Simon's book on geometric measure
  theory~\cite{simon-1984-lectures} (available for free online) has a
  couple of excellent chapters, one of which is an exposition of
  Allard's paper.
\item[Recent Varifold Work] Both Buet and
  collaborators~\cite{buet-2013-varifolds,buet-2014-quantitative,buet-2015-discrete,buet-2016-surface,buet-2016-varifold}
  and Charon and
  collaborators~\cite{charlier-2014-fshape,charon-2013-varifold,charon-2014-functional}
  have been digging into varifolds with an eye to applications. While
  these papers are a good start, there is still a great deal of
  opportunity for the use and further development of varifolds.  On
  the theoretical front, there is the work of Menne and
  collaborators~\cite{menne-2008-c2, menne-2009-some,
    menne-2010-sobolev, menne-2012-decay, menne-2014-weakly,
    kolasinski-2015-decay, menne-2016-pointwise, menne-2016-sobolev,
    menne-2017-concept, menne-2017-geometric}.  We want to call
  special attention to the recent introduction to the idea of a
  varifold that appeared in the December 2017 AMS
  Notices~\cite{menne-2017-concept}.
\item[Geometric Measure Theory I] The area underlying the ideas here are
  those from geometric measure theory. The fundamental treatise in the
  subject is still Federer's 1969 \emph{Geometric Measure
    Theory}~\cite{fed69} even though most people start by reading
  Morgan's beautiful introduction to the
  subject, \emph{Geometric Measure Theory: A Beginner's Guide}~\cite{morgan-2008-geometric} and Evans' \emph{Measure Theory and Fine Properties of Functions}~\cite{evans-1992-1}. Also recommended are Leon
  Simon's lecture notes~\cite{simon-1984-lectures}, Francesco Maggi's
  book that updates the classic Italian approach~\cite{mag12}, and
  Krantz and Parks' \emph{Geometric Integration
    Theory}~\cite{krantz-2008-geometric}. 
\item[Geometric Measure Theory II] The book by
  Mattila~\cite{mattila-1999-geometry} approaches the subject from the
  harmonic-analysis-flavored thread of geometric measure theory. Some
  use this as a first course in geometric measure theory, albeit one
  that does not touch on minimal surfaces, which is the focus of the
  other texts above. De Lellis' exposition \emph{Rectifiable Sets,
    Densities, and Tangent Measures}~\cite{de-lellis-2008-1} or Priess'
  1987 paper \emph{Geometry of Measures in $\R^n$: Distribution,
    Rectifiability, and Densities}~\cite{preiss-1987-geometry} is also
  very highly recommended.
\item[Jones' $\beta$] In addition to the papers cited in the text
  ~\cite{jon90,oki92,ler03,sch07}, there are related works by David and
  Semmes that we recommended. See for
  example~\cite{david-1993-analysis}. There is also the applied work
  by Gilad Lerman and his collaborators that is often inspired by
  Jones' $\beta$ and his own development of Jones' ideas
  in~\cite{ler03}. See
  also~\cite{chen-2009-spectral,zhang-2012-hybrid,zhang-2009-median,arias-2011-spectral}. See
  also the work by Maggioni and
  collaborators~\cite{little-2009-estimation,allard-2012-multiscale}.

\item[Multiscale Flat Norm] The flat norm was introduced by Whitney in
   1957~\cite{whitney-1957-geometric} and used to create a topology
  on currents that permitted Federer and Fleming, in their landmark
  paper in 1960~\cite{federer-1960-normal}, to obtain the existence of
  minimizers. In 2007, Morgan and Vixie realized that a variational
  functional introduced in image analysis was actually computing a
  multiscale generalization of the flat norm~\cite{morgan-2007-1}. The
  ideas are beginning to be explored in these papers
  \cite{vixie-2010-multiscale, van-dyke-2012-thin,
    ibrahim-2013-simplicial, vixie-2015-some, ibrahim-2016-flat,
    alvarado-2017-lower}.
\end{description}

\appendix

\section{Measures: A Brief Reminder}
\label{sec:ap.measure}

In this section we remind the reader of a handful of concepts used in
the text. 

\begin{description}

\item[Measure] One way to think of a measure is as a generalization of the familiar notions of length, area, and volume in a way that allows us to define how we assign ``size" to a given subset of a set $X$, the most common being that of $n$-dimensional Lebesgue measure $\mathcal{L}^n$. 

Formally, let $X$ be a nonempty set and $2^X$ be the collection of all subsets of $X$. A \textit{measure} is defined \cite{evans-1992-1} to be a mapping $\mu: 2^X\rightarrow[0,\infty]$ such that
\begin{enumerate}
\item $\mu(\emptyset)=0$ and\\ 

\item if $$A\subset\cup_{i=1}^\infty A_i,$$ then $$\mu(A)\leq\sum_{i=1}^\infty \mu(A_i).$$
\end{enumerate}

Note that in most texts, this definition is known as an \textit{outer measure}, but we use this definition with the advantage that we can  still ``measure" non-measurable sets. 

\item[$\mu$-measurable] A subset $S\subset X$ is called $\mu$-measurable if and only if it satisfies the Carath\'{e}odory condition for each set $A\subset X$: \[ \mu(A)=\mu(A\cap S) + \mu (A\setminus S).\]

\item[Radon Measure]  Let us define the Borel sets in $\R^n$ to be those sets that are derived from the set of all open sets in $\R^n$ through the operations of countable union, countable intersection, and set difference. Then a measure $\mu$ on $\R^n$ is a \textit{Radon} measure if \cite{evans-1992-1}
\begin{enumerate}
\item every Borel set is $\mu$-measurable; i.e. $\mu$ is a \textbf{Borel measure}.\\
\item  for each $A\subset\R^n$ there exists a Borel set $B$ such that $A\subset B$ and $\mu(A)=\mu(B);$ 
i.e. $\mu$ is \textbf{Borel regular.}\\
\item  for each compact set $K\in\R^n, \, \mu(K)<\infty$; i.e. $\mu$ is \textbf{locally finite.} \\
\end{enumerate}

\item[Hausdorff Measure] With this outer (radon) measure, we can
  measure $k$-dimensional subsets of $\R^n$ $(k\leq n)$. While it is
  true that $\mathcal{L}^n=\Hd^n$ for $n\in
  \mathbb{N}$ (see section 2.2 of \cite{evans-1992-1}), Hausdorff measure $\Hd^k$ is also defined for
  $k\in[0,\infty)$ so that even sets as wild as fractals are
  \textit{measurable} in a meaningful way.

To compute the $k$-dimensional Hausdorff measure of $A\subset\R^n$:
\begin{enumerate}
\item Cover $A$ with a collection of sets $\mathcal{E}=\{E_i\}_{i=1}^\infty$, where diam$(E_i)\leq d \ \forall i.$ 
\item Compute the $k$-dimensional measure of that cover: $$\mathcal{V}_{\mathcal{E}}^k(A)=\sum_i\alpha(k)\left(\frac{\text{diam}(E_i)}{2}\right)^k,$$ where $\alpha(k)$ is the $k$-volume of the unit $k$-ball. 
\item Define $\Hd_d^k(A)=\inf_{\mathcal{E}}\mathcal{V}_{\mathcal{E}}^k(A)$, where the infimum is taken over all covers whose elements 
have maximal diameter $d$. 
\item Finally, we define $\Hd^k(A)=\lim_{d\downarrow0}\Hd_d^k(A).$
\end{enumerate}

\begin{figure}[H]
  \centering
  \scalebox{1}{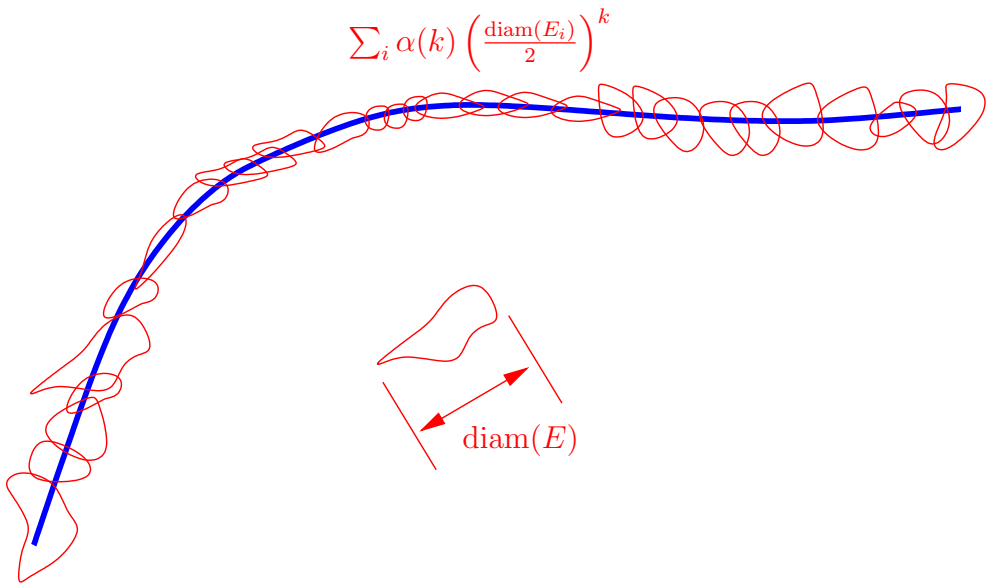}
  \caption{The Hausdorff Measure is derived from a cover of arbitrary sets.}
  \label{haus}
\end{figure}

\item[Approximate Tangent Plane] We present here an approximate
  tangent $k$-plane based on integration. (The one-dimensional version
  is of course an approximate tangent line.) We start with the fact
  that we can integrate functions defined on $\R^n$ over
  $k$-dimensional sets using $k$-dimensional measures $\mu$ (typically
  $\mathcal{H}^k$). We zoom in on the point $p$ through dilation of
  the set $F$:

\[F_\rho(p) = \{x\in\Bbb{R}^n \ | \;\;x=\frac{y-p}{\rho}+\text{ p for some }y\in F\}.\]

We will say that the set $F$ has an approximate tangent $k$-plane $L$
at $p$ if the dilation of $F_\rho(p)$ converges weakly to $L$;
i.e. if

\[\int_{F_\rho} \phi d\mu \; \rightarrow \; \int_L \phi d\mu \ \ \text{ as } \rho \rightarrow 0\]

for all continuously differentiable, compactly supported $\phi:\R^n \rightarrow \R.$

In the next two figures, we note that the solid green lines are the
level sets of $\phi$ while the dashed green line indicates the
boundary of the support of $\phi$. Note also that the $\rho$'s of 0.4,
0.1, and 0.02 are approximate.

\begin{figure}[H]
  \centering
  \scalebox{.8}{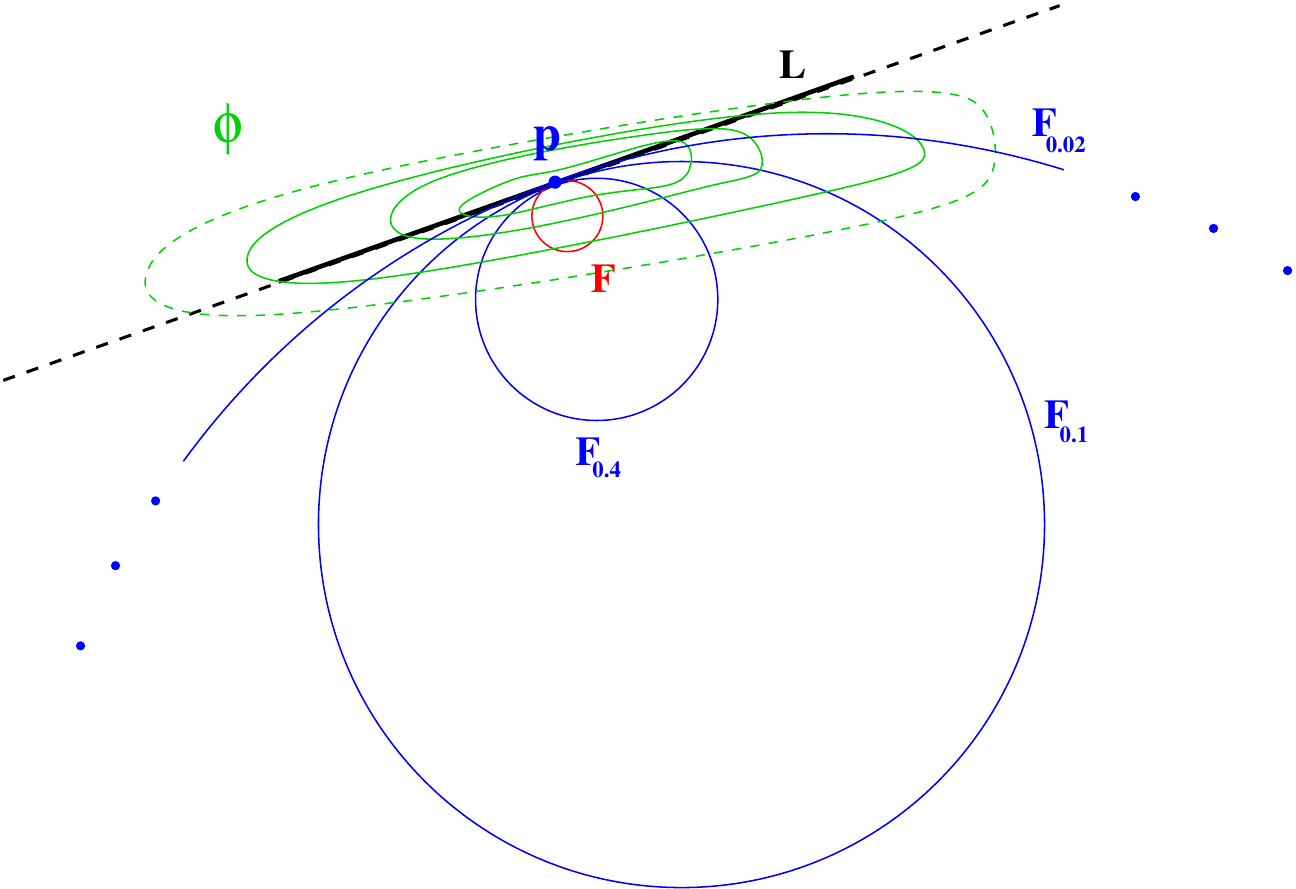}
  \caption{The case of 1-planes (lines) where $L$ is the weak limit of the dilations of $F$.}
  \label{app2}
\end{figure}

\begin{figure}[H]
  \centering
  \scalebox{.76}{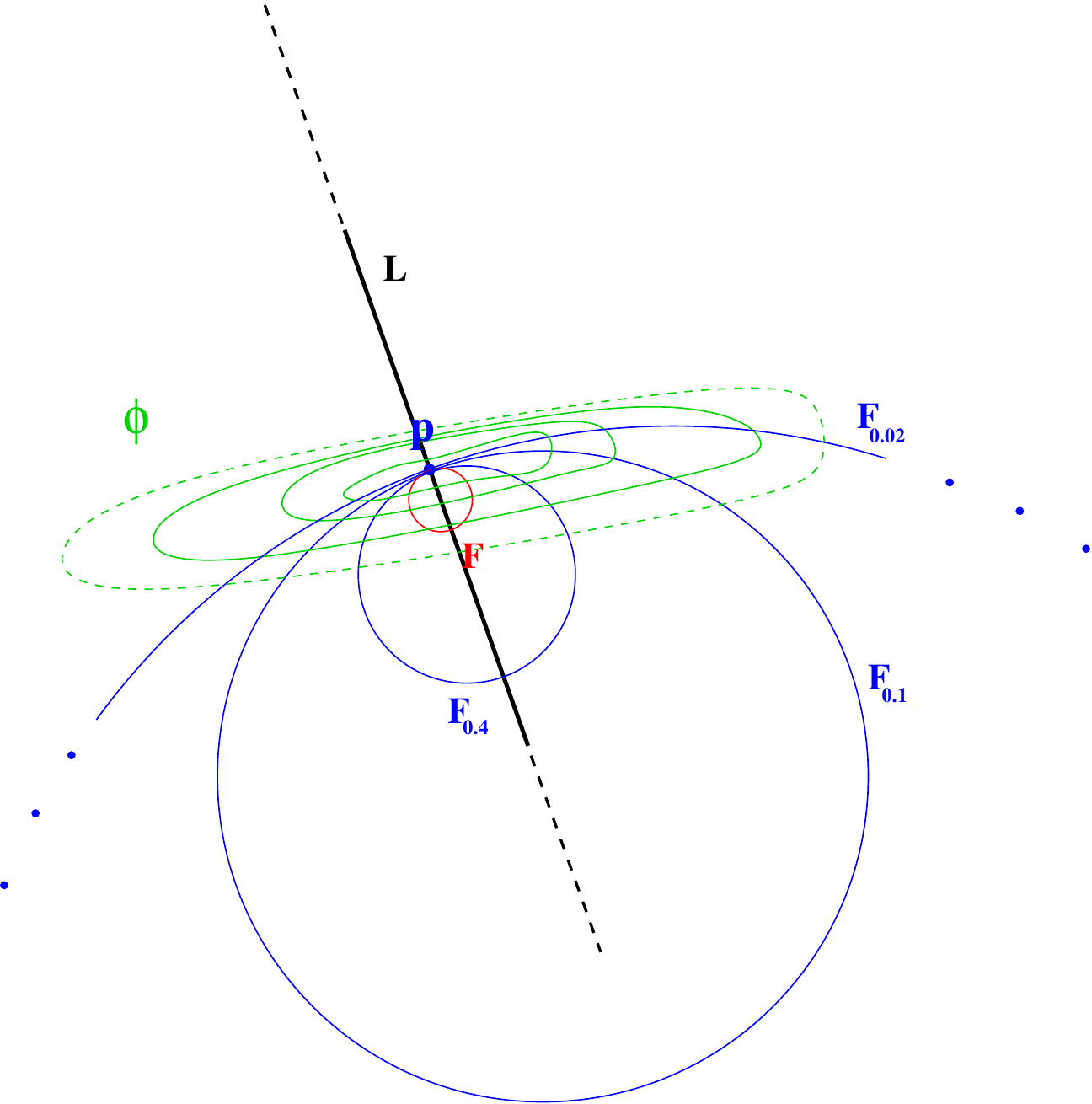}
  \caption{The case of 1-planes (lines) where $L$ is \textit{not} the weak limit of the dilations of $F$.}
  \label{app2wrong}
\end{figure}
\end{description}

\bibliographystyle{plain}
\bibliography{cubicalcovers}

\end{document}